\documentclass{article}

\usepackage{epsfig}
 \usepackage{epstopdf}
\usepackage[T1]{fontenc}
\usepackage{geometry}
\usepackage{amsbsy,amsmath,latexsym,amsfonts, epsfig, color, authblk, amssymb, graphics, bm}
\usepackage{epsf,slidesec,epic,eepic}
\usepackage{fancybox}
\usepackage{fancyhdr}
\usepackage{setspace}
\usepackage{nccmath}
\usepackage[colorlinks,linkcolor=black,anchorcolor=black,citecolor=black,hyperindex,CJKbookmarks]{hyperref}
\newenvironment{proof}{{\noindent \bf Proof.}}{\hfill$\Box$\medskip}

\newtheorem{theorem}{Theorem}[section]
\newtheorem{corollary}[theorem]{Corollary}
\newtheorem{lemma}[theorem]{Lemma}

\def \b{\beta}

\def \l{\lambda}

\def \G{\Gamma}

\def \d{\delta}
\def \e{\epsilon}

\def \w{\omega}

\def \S{\Sigma}

\def \N{\mathbb{N}}

\def \A{\mathcal{A}}

\begin{document}

\title{ The topological property of the irregular sets on the lengths of basic intervals in beta-expansions\footnotetext {* Corresponding author}
}
\author{  Lixuan Zheng$^\dag$, Min Wu$^\dag$ and Bing Li$^{\dag, *}$\\
\small \it $\dag$ Department of Mathematics\\
\small \it South China University of Technology\\
\small \it Guangzhou 510640, P.R. China\\
\small \it E-mails: z.lixuan@mail.scut.edu.cn, wumin@scut.edu.cn and
scbingli@scut.edu.cn}
\date{5th March,2015}
\date{}
\maketitle
\begin{center}
\begin{minipage}{120mm}{\small {\bf Abstract.} Let $\beta > 1$ be a real number. A basic interval of order $n$ is a set of real numbers in $(0,1]$ having the same first $n$ digits in their $\beta$-expansion which contains $x\in(0,1]$, denote by $I_n(x)$ and write the length of $I_n(x)$ as  $|I_n(x)|$. In this paper, we prove that the extremely irregular set containing points $x \in [0, 1]$ whose upper limit of $\frac{-\log_\beta |I_n(x)|}{n}$ equals to $1+\l(\beta)$ is residual for every $\lambda(\beta)>0$, where $\l(\beta)$ is a constant depending on $\beta$.}
\end{minipage}
\end{center}

\vskip0.5cm

\section{Introduction}
Fix a real number $\b > 1$. Define the \emph{$\b$-transformation} $T_{\b}:(0,1] \rightarrow (0,1]$ by $$T_{\b}x = \b x-\lceil\b x\rceil + 1,$$ where $\lceil x \rceil$ stands for the smallest integer no less than $x$. It is well known \cite{R} that, by the iteration of $T_{\b}$, every $x \in (0, 1]$ can be  written as:
\begin{equation}\label{1.1}
x=\frac{\epsilon_1(x,\b)}{\b}+\cdots+\frac{\e_n(x,\b)+T_{\b}^n
x}{\b^n}=\sum_{n=1}^\infty \frac{\e_n(x,\b)}{\b^n},
\end{equation}
where, for each $n\geq1$,$$\e_n(x,\b)=\lceil \b T_{\b}^{n-1}x
\rceil-1$$ is called the \emph{$n$-th digit} of $x$. We identify $x$
with its digit
sequence$$\e(x,\b):=(\e_1(x,\b),\ldots,\e_n(x,\b),\ldots)$$and the digit sequence $\e(x,\b)$ is said to be the \emph{$\b$-expansion} of $x$. Sometimes we write $\e_n(x)$ instead of $\e_n(x,\beta)$ if $\beta$ is fixed.

For an admissible word $(\e_1,\ldots,\e_n)$, i.e.,\ a prefix of the
digit sequence for some $x \in (0, 1]$, we define the \emph{basic
interval of order $n$} which is denoted by $I(\e_1,\ldots,\e_n)$ as
$$I(\e_1 ,\ldots,\e_n):= \{x \in (0,1]: \e_j(x,\b)=\e_j,\ {\rm for\ all}\  1 \leq j \leq n\}.$$We write the basic interval of order $n$
containing $x$ as $I_n(x)$. A simple fact of the basic interval $I_n(x)$ is that it is a left-open and right-closed interval, see Lemma \ref{i} for more details. We write the length of $I_n(x)$ as $|I_n(x)|$.

The $\b$-expansion of the unit $1$ has played an important role not only in researching the dynamical properties of the orbit of $1$, but also in estimating the length of $I_n(x)$ (\cite{BW}, see also \cite{LW,BT}). Let$$1=\frac{\e_1^\ast}{\b}+\cdots+\frac{\e_n^\ast}{\b^n}+\cdots$$be the $\b$-expansion of the unit $1$. For each integer $n \geq 1$,
denote by $t_n = t_n(\b)$ the maximal length of consecutive zeros just follow the $n$-th digit of the $\b$-expansion of $1$. That is,
\begin{equation}\label{1.3}
t_n=t_n(\b) := \max\{k \geq 0:\e_{n+1}^\ast = \e_{n+2}^\ast =
\cdots= \e_{n+k}^\ast=0\}.
\end{equation}
Define
\begin{equation}\label{1.4}
\l(\b) = \limsup_{n\rightarrow \infty}\frac{\G_n(\b)}{n},
\end{equation}
where$$\G_n = \G_n(\b) := \max_{1 \leq k \leq n}t_k(\b).$$

The estimation on the lengths of basic intervals is very useful to study the fractals in $\b$-expansion such as the multifractal spectra for the recurrence rates of $\b$-transformations \cite{BB}, the Diophantine properties of the orbits in $\b$-expansions \cite{BW} and so on. Fan and Wang \cite{AB} established a relationship between the length of $I_n(x)$ and the $\b$-expansion of $1$ and gave a way to calculate the length of $I_n(x)$, see Theorem \ref{th1.1} for more details. Furthermore, they introduced and studied the quantities which describe the growth
of the length of $I_n(x)$. More precisely, for any $x \in (0,1]$, define the lower and upper density at $x$ for $\b$-expansion respectively, i.e.,\
$$\underline{D}(x) = \liminf_{n\rightarrow
\infty}\frac{-\log_\b | I_n(x)| }{n}\  \  \  {\rm and}\   \ \overline{D}(x) = \limsup_{n\rightarrow \infty}\frac{-\log_\b
|I_n(x)|}{n}.$$
It is known \cite{LW} that for any $x \in (0, 1]$, we have
\begin{equation}\label{up}
\underline{D}(x) = 1,\ 1 \leq \overline{D}(x) \leq 1+\l(\b).
\end{equation}
Shannon-McMillan-Breiman Theorem applied to Parry measure \cite{PA} gives that $\underline{D}(x) = \overline{D}(x)=1$ for Lebesgue almost all $x\in(0,1]$. A special case is $\l(\b)=0$, we obtain that $\underline{D}(x) = \overline{D}(x)=1$ which indicates that the limit of $\frac{-\log_\b
|I_n(x)|}{n}$ exists for all $x\in (0,1]$. The set of such $\b$ with $\l(\b)=0$ in $(1,+\infty)$ is of full Lebesgue measure \cite{BT}. Then people turn to focus on the exceptional set with respect to the upper density for the case $\l(\b)>0$. From now on, unless other indicated, we only consider the case $\l(\b)>0$. For any $1 < \d \leq 1 + \l(\b)$, define $$E_\d = \{x \in (0, 1] : \overline{D}(x) = \d \},$$which is a Lebesgue null set. The points in $E_\d$ with $1<\d \leq 1+\l(\b)$ are said to be \emph{$\d$-irregular}. Fan and Wang \cite{AB} showed that
\begin{equation}\label{th1.3}
\dim_{\rm H}E_\d=\frac{\l(\b)+1-\d}{\d \cdot \l(\b)},
\end{equation}for every $1<\d \leq 1+\l(\b)$, where $\dim_{\rm H}$ denotes the Hausdorff
dimension.

One natural question is how large the sets $E_\d$ are in the sense of topology. Motivated by this, we devote to showing that the extremely irregular set $E_{1+\l(\b)}$ is residual, see Theorem \ref{r1} for more details.

To state our results, we begin by introducing some notation. For all $x \in (0,1]$, let $A(D(x))$ denote the set of accumulation points of $
\frac{-\log_\b |I_n(x)|}{n} $ as $n \nearrow \infty$, that is,
$$A(D(x)) = \left\{y \in [1,1+\l(\b)]:\lim_{k \rightarrow \infty}\frac{-\log_\b |I_{n_k}(x)|}{n_k}  = y\ {\rm for\ some\ } \{n_k\}_{k\geq1} \nearrow \infty\right\}.$$

For an integer $n\geq 1$, denote by $k_n^\ast (x)$ the largest length of the suffixes of $(\e_1(x,\b),\ldots,\e_n(x,\b))$ agreeing
with the prefix of the digit sequence of the unit $1$. In other words,
\begin{equation}\label{kn}
k_n^\ast (x)= k_n^\ast (x,\b):= \inf\{k \geq 0:
(\e_{k+1}(x,\b),\ldots,\e_n(x,\b))=(\e_1^\ast,\ldots,\e_{n-k}^\ast)\}.
\end{equation}
Define
\begin{equation}\label{lx}
\tau(x) = \tau(x,\b):=\limsup_{n \rightarrow
\infty}\frac{t_{n-k_n^\ast(x)}}{n}.
\end{equation}
where $t_n$ is defined as (\ref{1.3}). Since $n-k_n^\ast(x) \leq n$ and by the definition of $\G_n$, we know that $t_{n-k_n^\ast(x)} \leq \G_n$ which implies $\tau(x)\leq \l(\b)$ for any $x \in (0,1]$.

We shall prove that the set $A(D(x))$ is always a closed interval. More precisely, we will show:
\begin{theorem}\label{closed interval}
Let $x \in (0,1]$. Then, $A(D(x))=[1,1+\tau(x)].$
\end{theorem}

The above theorem indicates that the set $E_\d$ can be written as $\{x \in (0,1]:A(D(x))=[1,\d]\}.$  An extreme case is that $\d=1+\l(\b)$ which means that the accumulation points of $\frac{-\log_\b | I_n(x)| }{n}$ can contain any possible value in $[1,1+\l(\b)]$. So the points in $E_{1+\l(\b)}$ is said to be extremely irregular and for convenience, we write $E=E_{1+\l(\b)}$.

The following main theorem illustrates that the extremely irregular set $E$ is large from a topological viewpoint for every $\l(\b)>0$. An important point we should notice is that not only the Lebesgue measure of $E$ is $0$, but also its Hausdorff dimension is $0$. This result is somewhat similar to Olsen's work \cite{O} on the extremely non-normal number.
\begin{theorem}\label{r1}
Let $\b>1$ with $\l(\b)>0$. Then the set $E$ is residual, in other words, $[0,1]\setminus E $ is of the first category. In particular, the set $E$ is of the second category.
\end{theorem}

Noting that $E \cap E_\d=\emptyset$ for all $1<\d<1+\l(\b)$, the following corollary is immediate.
\begin{corollary}
Let $\b>1$ with $\l(\b)>0$, then $E_\d$ is of the first category for every $1<\d<1+\l(\b)$.
\end{corollary}

Furthermore, Theorem \ref{r1} implies that $\overline{E}=[0,1]$ where $\overline{E}$ denote the closure of $E$, so $\dim_{\rm{B}}E=\dim_{\rm{B}}\overline{E}=1$. Hence, we easily get the following corollary which implies that $0=\dim_{\rm{H}}E < \dim_{\rm{B}}E=1$.
\begin{corollary}\label{r3}
If $\l(\b)>0$, then $\dim_{\rm{B}}E=1.$
\end{corollary}

Let $D$ be an \emph{irregular set} which contains the points of $x\in(0,1]$ whose limit of $\frac{-\log_\b | I_n(x)| }{n}$ does not exists, i.e.,
$$D = \{x \in (0,1]:\underline{D}(x) < \overline{D}(x)\}.$$  The set of irregular points is negligible from the measure-theoretical point of view \cite{TD}. For every $ 1 < \d \leq 1 + \l(\b)$, we have $E_\d \subset D$, so $\dim_{\rm H}D \geq \dim_{\rm H} E_{\d} = \frac{\l(\b)+1-\d}{\d \cdot \l(\b)}\rightarrow 1$ as $\d \rightarrow 1$ (\ref{th1.3}) which implies that $$\dim_{\rm H}D =\dim_{\rm P}D=\dim_{\rm B}D=1,$$ where $\dim_{\rm P}$ and $\dim_{\rm B}$ denote the packing and boxing dimension respectively, see \cite{FE} for more details. Thus, we obtain that the set $D$ has full dimension if $\l(\b)>0$, i.e.,\ the set $D$ can be large from the viewpoint of dimension theory. The next result shows that $D$ is large from a topological point of view as well which follows immediately from $E\subset D$.
\begin{corollary}
Let $\b>1$ with $\l(\b)>0$, then $D$ is residual, therefore $D$ is of second category.
\end{corollary}

In fact, there exist some irregular sets with zero measure, but residual, which indicates that such sets can be large in the topological sense. For example, the sets of some kinds of irregular points associated with integer expansion are shown to be residual \cite{PT,H,O}. Baek and Olsen \cite{BO} proved the set of extremely non-normal points of self-similar set is of residue. Madritsch \cite{M} extended and generalized these results to Markov partitions. Also, Madritsch and Petrykiewicz\cite{MI} showed that the non-normal numbers in dynamical system fulfilling the specification property are residual. However, in the study of non-normal number, the frequencies of digits and blocks were investigated. In this paper, the upper density $\overline{D}(x)$ cannot be expressed as some frequencies, and thus is a new object of research.

\section{Preliminaries}
In this section, we will recall some basic facts of $\b$-expansions and fix some notation. For more properties of
$\b$-expansions see \cite{B,HF,WP} and references therein.

The typical $\b$-transformation is given by $$T(x):=\b x - \lfloor\b x\rfloor, 0 \leq x < 1, $$where $\lfloor x \rfloor$ denotes the largest integer which is less than or equal to $x$. The transformation $T_{\b}$ adopted in this paper in order to ensure that every $x \in (0, 1]$ has an infinite series expansion, i.e.,\ $\e_n(x, \b) \neq 0$ for infinitely many $n \in \N$. This is because $T_{\beta}(x)$ is strictly larger than $0$. As a mater of fact, $\b$-expansions under the above two transformation coincide except at the points with a finite expansion under the algorithm
$T$.

From the definition of $T_\b$, it is clear that, for an integer $n \geq 1$, the $n$-th digit $\e_n(x,\b)$ of $x$ belongs to the
alphabet $\A=\{0,\ldots,\lceil\b\rceil-1\}$. What we should note here is that not all sequences $\e \in \A^\N$ are the $\b$-expansion of some $x\in(0,1]$. This leads to the notation of \emph{$\b$-admissible sequence}.

A word $(\e_1,\ldots,\e_n)$ is said to be \emph{admissible} with respect to the base $\beta$ if there
exists an $x \in (0, 1]$ such that the $\beta$-expansion of $x$ satisfies $\e_1(x,\beta)=\e_1,\ldots,\e_n(x,\beta)= \e_n$. An infinite digit sequence $(\e_1,\ldots,\e_n,\ldots)$ is called admissible if there exists an $x \in (0, 1]$ has the $\b$-expansion as $(\e_1,\ldots,\e_n,\ldots)$.

Let $\S_\b^n$ denote the family of all $\b$-admissible words with length $n$, i.e.,\ $$\S_\b^n=\{(\e_1,\ldots,\e_n)\in \A^n:
\exists\ x \in (0,1], {\rm such\ that\ }\e_j(x,\b)=\e_j, \forall\ 1
\leq j \leq n\}.$$ Let $\S_\b^\ast$ be the family  of all $\b$-admissible words with finite length, i.e.,\
$$\S_\b^\ast=\bigcup_{n=0}^\infty \S_\b^n.$$Let $\S_\b$ be the family  of all infinite $\b$-admissible sequences,
i.e.,\ $$\S_\b=\{(\e_1,\e_2,\ldots)\in \A^\N: \exists\ x \in (0,1],
{\rm such\ that\ }\e_j(x,\b)=\e_j, \forall\ j \geq 1\}.$$

We endow the space $\A^\N$ with the \emph{lexicographical order $<_{\rm{lex}}$} as
follows: $$(\varepsilon_1, \varepsilon_2,\ldots)<_{\rm{lex}}(\varepsilon'_1, \varepsilon'_2,\ldots)$$if there exists an integer $k \geq 1$ such that, for all $1 \leq j < k$, $\varepsilon_j=\varepsilon'_j$ but $\varepsilon_k<\varepsilon'_k$. The symbol $\leq_{\rm{lex}}$ means $=$ or $<_{\rm{lex}}$.

A characterization of the admissibility of a sequence which relies heavily on the $\beta$-expansion of $1$ is given by Parry \cite{WP} as the following theorem.
\begin{theorem}[Parry\cite{WP}]\label{P}
Fix $\beta > 1$, for every $n\geq 1$,$$(\e_1,\ldots,\e_n)\in \Sigma_\beta^n \Longleftrightarrow \sigma ^i\omega \leq_{\rm{lex}} (\e_1^\ast,\ldots,\e_{n-i}^\ast)\ for\ all\ i \geq 1,$$ where $\sigma$ is the shift operator such that $\sigma\omega=(\omega_2,\omega_3,\ldots).$
\end{theorem}

Now we give a simple fact on the basic intervals, readers can refer to \cite{AB} for more details.

\begin{lemma}\label{i}
Let $\e =(\e_1,\ldots,\e_n)\in \S_\b^n$ with $n\geq 1$. We have $I(\e_1,\ldots,\e_n)$ is a left-open and right-closed interval with $\frac{\e_1}{\b}+\cdots+\frac{\e_n}{\b^n}$ as its left endpoint.
\end{lemma}


The notation of full intervals is vital to give the estimation of $|I_n(x)|$ in this paper, now we give the definition and state some simple facts on the full intervals. A basic interval $I(\e_1,\ldots, \e_n)$ is said to be \emph{full} if its length verifies $$|I(\e_1, \ldots, \e_n)| = \b^{-n}.$$

Fan and Wang  \cite{AB} gave serval characterizations and properties of full intervals as follow.
\begin{theorem}[\cite{AB}]\label{lem3.2}
Let $\e=(\e_1, \e_2,\ldots,\e_n)\in \S_\b^n$ with $n \geq 1$.\\
(1)The basic interval $I(\e_1,\ldots, \e_n)$ is a full interval if and only if for any $m \geq 1$ and any $\e'= (\e'_1,\ldots, \e'_m)\in \S_\b^m$, the concatenation $\e \ast \e'=(\e_1, \ldots,\e_n,\e'_1,\ldots, \e'_m)$
is admissible. \\
(2) If $(\e_1,\ldots,\e_{n-1},\e'_n)$ with $\e'_n \neq 0$ is admissible, then $I(\e_1,\ldots,\e_{n-1},\e_n)$ is full for any $0 \leq \e_n < \e'_n$.\\
(3) If $I(\e_1,\ldots,\e_n)$ is full, then for any $(\e'_1,\ldots,\e'_m)\in \S_\b^m$ , we have $$|I(\e_1,\ldots, \e_n,\e'_1,\ldots,\e'_m)| = |I(\e_1,\ldots,\e_n)| \cdot |I(\e'_1,\ldots,\e'_m)|.$$
(4)The basic intervals $I(\e_1,\ldots,\e_n,0^{\G_n+1})$ and $I(\e_1^\ast,\ldots,\e_n^\ast,0^{t_n+1})$ are full, where $0^\ell=\underbrace{0,\ldots,0}_\ell$.
\end{theorem}

The following inequalities on the estimation of the lengths of basic intervals will be used which follow from Theorem \ref{lem3.2}(4)(see also \cite{LW}). For every admissible word $(\e_1,...,\e_n)$, we have
\begin{equation}\label{1.5}
\b^{-(n+\G_n + 1)} \leq |I(\e_1,\ldots,\e_n)| \leq \b^{-n},
\end{equation}
\begin{equation}\label{1.6}
\b^{-(n+t_n + 1)} \leq |I(\e_1^\ast,\ldots,\e_n^\ast)| \leq
\b^{-(n+t_n)}.
\end{equation}

Moreover, the following theorem in \cite{AB} gives a way to evaluate the length of an arbitrary basic interval $I(\e_1,\ldots,\e_n)$ by
comparing the suffixes of $(\e_1,\ldots,\e_n)$ with the prefixes of $\b$-expansion of the unit $1$.
\begin{theorem}[\cite{AB}]\label{th1.1}
Let $\e=(\e_1, \ldots,\e_n)\in \S_\b^n$ with $n \geq 1$.
Let$$k_n^\ast=k_n^\ast(\e) = \inf\{k \geq 0:(\e_{k+1},\ldots,\e_n)=(\e_1^\ast, \ldots, \e_{n-k}^\ast)\}.$$Then the length of $I(\e_1,\ldots,\e_n)$ satisfies $$|I(\e_1, \ldots,\e_n)| =\b^{-k_n^\ast}|I(\e_1^\ast,\ldots, \e_{n-k_n^\ast}^\ast)|.$$
\end{theorem}

\section{ Proof of Theorem \ref{closed interval}}
We will prove Theorem \ref{closed interval} in this section and before doing that, we first give a lemma on the upper density.

\begin{lemma}\label{sup}
Let $x \in (0,1]$. Then $\overline{D}(x)=1+\tau(x).$
\end{lemma}
\begin{proof}
On the one hand, by the definition of $k_n^\ast(x)$(\ref{kn}) and Theorem
\ref{th1.1}, we have
\begin{equation}\label{kn1}
|I_n(x)|=\b^{-k_n^\ast(x)}|I(\e_1^\ast,\ldots,\e_{n-k_n^\ast(x)}^\ast)|.
\end{equation}
It immediately follows from Theorem \ref{lem3.2}(4) that
\begin{equation}\label{kn2}
|I(\e_1^\ast,\ldots,\e_{n-k_n^\ast(x)}^\ast)| \geq
|I\big(\e_1^\ast,\ldots,\e_{n-k_n^\ast(x)}^\ast,0^{t_{n-k_n^\ast(x)}+1}\big)|=\b^{-\big(n-k_n^\ast(x)+t_{n-k_n^\ast(x)}+1\big)}.
\end{equation}

Combining (\ref{kn1}) and (\ref{kn2}), we get
\begin{equation}\label{1}
\overline D(x) \leq
\limsup\limits_{n\rightarrow\infty}\frac{n+t_{n-k_n^\ast(x)}+1}{n}
= 1+\tau(x).
\end{equation}

One the other hand, we need to find a sequence $\{n_i\}_{i\geq 1}$ satisfying $\overline{D}(x)=1+\tau(x).$ In fact, by the definition of $\tau(x)$, we can find a sequence $\{n_i\}_{i\geq 1}$ such that  $\tau(x)=\lim\limits_{i \rightarrow
\infty}\frac{t_{n_i-k_{n_i}^\ast(x)}}{n_i}$. In addition, (\ref{1.6}) gives that
\begin{equation}\label{kn3}
\b^{-\left(n_i-k_{n_i}^\ast(x)+t_{n_i-k_{n_i}^\ast(x)}+1\right)}
\leq |I(\e_1^\ast,\ldots\e_{n-k_n^\ast(x)}^\ast)| \leq
\b^{-\big(n_i-k_{n_i}^\ast(x)+t_{n_i-k_{n_i}^\ast(x)}\big)}.
\end{equation}

Consequently, applying (\ref{kn1}) and (\ref{kn3}), we deduce that
$$\frac{n_i+t_{n_i-k_{n_i}^\ast(x)}}{n_i} \leq \frac{-\log_\b|I_{n_i}(x)|}{n_i} \leq \frac{n_i+t_{n_i-k_{n_i}^\ast(x)}+1}{n_i},$$
that is,
\begin{equation}\label{2}
\lim_{i\rightarrow\infty}\frac{-\log_\b|I_{n_i}(x)|}{n_i}=1+\tau(x).
\end{equation}

Combination of (\ref{1}) and (\ref{2}) gives the desired result.
\end{proof}\\
\textbf{Proof of Theorem \ref{closed interval}}  We divide the proof into two cases by showing that $A(D(x))=\{1\}$ when $\tau(x)=0$ and $A(D(x)) = [1,1+\tau(x)]$ when $\tau(x)>0$.

Case \uppercase\expandafter{\romannumeral1}: $\tau(x)=0$. Note that $\underline{D}(x)=1$ by (\ref{up}) and
$\overline{D}(x)=1+\tau(x)$ by Lemma \ref{sup}, then
$\underline{D}(x)=\overline{D}(x)=1$. So $A(D(x))=\{1\}$.

Case \uppercase\expandafter{\romannumeral2}: $\tau(x)>0$. Since  $\underline{D}(x)=1$ and
$\overline{D}(x)=1+\tau(x)$, for any $1<a<1+\tau(x)$, we can choose an increasing sequence
$\{n_k\}_{k \geq 1}$ tending to $\infty$ as $k\rightarrow \infty$
such that
$$\frac{-\log_\b |I_{n_k+1}(x)|}{n_k+1} \leq a \leq \frac{-\log_\b
|I_{n_k}(x)|}{n_k}.$$  Noting that
$|I_{n_k}(x)| \geq |I_{n_k+1}(x)|$, we know that
$$\frac{-\log_\b |I_{n_k}(x)|}{n_k}\leq\frac{-\log_\b
|I_{n_k+1}(x)|}{n_k} = \frac{-\log_\b
|I_{n_k+1}(x)|}{n_k+1}\cdot\frac{n_k+1}{n_k}.$$ Therefore,
$$\frac{-\log_\b |I_{n_k+1}(x)|}{n_k+1} \leq a \leq \frac{-\log_\b
|I_{n_k+1}(x)|}{n_k+1}\cdot\frac{n_k+1}{n_k},$$ which implies
$\lim\limits_{k\rightarrow \infty}\frac{-\log_\b
|I_{n_k+1}(x)|}{n_k+1} = a$. Thus $[1,1+\tau(x)]\subset A(D(x))$.

Moreover, $\underline{D}(x)=1$ and $\overline{D}(x)=1+\tau(x)$ indicate that $A(D(x)) \subset [1,1+\tau(x)]$.

Therefore, $A(D(x)) = [1,1+\tau(x)]$.$\hfill\Box$

\section{Proof of Theorem \ref{r1}}

To prove  Theorem \ref{r1}, we first introduce some notation in symbolic space.

For  $\e = (\e_1,\e_2,...,\e_n) \in \A^n$ and a positive
integer $m$ with $m \leq n$, or for $\e = (\e_1,\e_2,...,\e_n,...)
\in \A^{\N}$ and a positive integer $m$, let $\e|_m =
(\e_1,\e_2,...,\e_m)$.

Recall that $(\e_1^\ast,\ldots,\e_n^\ast,\ldots)$ is the $\b$-expansion of $1$ and $t_n$ is defined as (\ref{1.3}), we give a property of the full intervals as the following lemma.
\begin{lemma}\label{in}
Let $k\geq 1$ be an integer. If $I(\e_1,\ldots,\e_n)$ is full, we have
$$\overline{I(\e_1,\ldots,\e_n,\e_1^\ast,\ldots,\e_k^\ast,0^{t_k+1})}\subset {\rm int}(I(\e_1,\ldots,\e_n))$$where
${\rm int}(I(\w))$ denotes the interior of $I(\w) \subset [0,1]$.
\end{lemma}
\begin{proof}
If $I(\e_1,\ldots,\e_n)$ is full, we get that $(\e_1,\ldots,\e_n)$ can concatenate any admissible word by Theorem \ref{lem3.2}(1). By the definition of $t_n$, we have $(\e_1,\ldots,\e_n,\e_1^\ast,\ldots,\e_k^\ast,0^{t_k+1})\in \S_\b^\ast$. Now we only need to show that the left and right endpoints of $I(\e_1,\ldots,\e_n,\e_1^\ast,\ldots,\e_k^\ast,0^{t_k+1})$ lie in ${\rm int}(I(\e_1,\ldots,\e_n))$ respectively.

Case \uppercase\expandafter{\romannumeral1}: Since $\e_1^\ast=\lceil\b \rceil-1 \geq 1$, we have $$\frac{\e_1}{\b}+\cdots+\frac{\e_n}{\b^n}+\frac{\e_1^\ast}{\b^{n+1}}+\cdots+\frac{\e_k^\ast}{\b^{n+k}} > \frac{\e_1}{\b}+\cdots+\frac{\e_n}{\b^n}.$$
This inequality and $I(\e_1,\ldots,\e_n,\e_1^\ast,\ldots,\e_k^\ast,0^{t_k+1})\subset I(\e_1,\ldots,\e_n) $ imply that the left endpoint of $I(\e_1,\ldots,\e_n,\e_1^\ast,\ldots,\e_k^\ast,0^{t_k+1})$ belongs to ${\rm int}(I(\e_1,\ldots,\e_n))$ by Lemma \ref{i}.

Case \uppercase\expandafter{\romannumeral2}: By the definition of $t_k$, we have $(\e_1,\ldots,\e_n,\e_1^\ast,\ldots,\e_k^\ast,0^{t_k},1)$ is admissible. Besides, the fullness of $I(\e_1,\ldots,\e_n)$ ensures that $(\e_1,\ldots,\e_n,\e_1^\ast,\ldots,\e_k^\ast,0^{t_k},1)\in \S_\b^\ast$ by Theorem \ref{lem3.2}(1). The same argument as Case \uppercase\expandafter{\romannumeral1} gives that the left endpoint of $I(\e_1,\ldots,\e_n,\e_1^\ast,\ldots,\e_k^\ast,0^{t_k},1)$ belongs to ${\rm int}(I(\e_1,\ldots,\e_n))$, that is,
$$x_0:=\frac{\e_1}{\b}+\cdots+\frac{\e_n}{\b^n}+\frac{\e_1^\ast}{\b^{n+1}}+\cdots+\frac{\e_k^\ast}{\b^{n+k}}+\frac{1}{\b^{n+k+t_k+1}} \in {\rm int}(I(\e_1,\ldots,\e_n)).$$

Now we prove that $x_0$ is the right endpoint of $I(\e_1,\ldots,\e_n,\e_1^\ast,\ldots,\e_k^\ast,0^{t_k+1})$. As a matter of fact, for every $x\in I(\e_1,\ldots,\e_n,\e_1^\ast,\ldots,\e_k^\ast,0^{t_k+1}) $, we easily get that
$$\begin{array}{rcl}
x & = & \frac{\e_1}{\b}+\cdots+\frac{\e_n}{\b^n}+\frac{\e_1^\ast}{\b^{n+1}}+\cdots+\frac{\e_k^\ast}{\b^{n+k}}+\frac{\e_{n+k+t_k+2}(x,\beta)}{\b^{n+k+t_k+2}} +\cdots\\
& \leq & \frac{\e_1}{\b}+\cdots+\frac{\e_n}{\b^n}+\frac{\e_1^\ast}{\b^{n+1}}+\cdots+\frac{\e_k^\ast}{\b^{n+k}}+\frac{1}{\b^{n+k+t_k+1}}=x_0,
\end{array}$$
Furthermore, recall that $1=\frac{\e_1^\ast}{\b}+\frac{\e_2^\ast}{\b^2}+\cdots$, we obtain that
$$\begin{array}{rcl}
x_0 & = & \frac{\e_1}{\b}+\cdots+\frac{\e_n}{\b^n}+\frac{\e_1^\ast}{\b^{n+1}}+\cdots+\frac{\e_k^\ast}{\b^{n+k}}+\frac{1}{\b^{n+k+t_k+1}} \\
& = & \frac{\e_1}{\b}+\cdots+\frac{\e_n}{\b^n}+\frac{\e_1^\ast}{\b^{n+1}}+\cdots+\frac{\e_k^\ast}{\b^{n+k}}+\frac{1}{\b^{n+k+t_k+1}}(\frac{\e_1^\ast}{\b}+\cdots+ \frac{\e_n^\ast}{\b^n}+\cdots)\\
&=& \frac{\e_1}{\b}+\cdots+\frac{\e_n}{\b^n}+\frac{\e_1^\ast}{\b^{n+1}}+\cdots+\frac{\e_k^\ast}{\b^{n+k}}+\frac{\e_1^\ast}{\b^{n+k+t_k+2}}+ \frac{\e_2^\ast}{\b^{n+k+t_k+3}}+\cdots \\
&\in & I(\e_1,\ldots,\e_n,\e_1^\ast,\ldots,\e_k^\ast,0^{t_k+1}).
\end{array}$$
The last relationship follows from the criterion of admissibility (Theorem \ref{P}), so $x_0$ is the right endpoint of the basic interval $I(\e_1,\ldots,\e_n,\e_1^\ast,\ldots,\e_k^\ast,0^{t_k+1})$ and it belongs to ${\rm int}(I(\e_1,\ldots,\e_n))$.
\end{proof}

The notion of residual set is usually used to describe a set being large in the topological sense. Recall that in a metric space $X$, a set $R$ is said to be \emph{residual} if its complement is of the first category. Moreover, in a complete metric space a set is residual if it contains a dense $G_\d$ set, see \cite{JC}. Hence, in order to prove Theorem\ref{r1}, it suffices to construct a set $U \subset [0,1]$ verifying the following three conditions:

(1) $U \subset E;$

(2) $U$ is dense in $[0,1]$;

(3) $U$ is a $G_\d$ set.

Now we devote to constructing a set $U$ with the desired properties. From \cite{AB}, we know that $\l(\b)$ can also be written as $$\l(\b) = \limsup_{n\rightarrow \infty}\frac{t_n}{n},$$where $t_n$ is defined as (\ref{1.3}). For each $k\geq 1$, recall that $\Sigma_{\beta}^k$ is the set of all admissible words of length $k$. Define $$U:=\bigcap_{n=1}^{\infty}\bigcup_{k=n}^{\infty}\bigcup_{(\e_1,\ldots,\e_k)\in \Sigma_{\beta}^k}\rm{int}\left(I(\e_1,\ldots,\e_k,0^{\Gamma_k+1},\omega_1,\ldots,\omega_k)\right),$$
where \begin{equation}\label{wk}
\omega_k=(\e^\ast_1,\ldots,\e^\ast_{m_k},0^{t_{m_k}+1}),
\end{equation} and $m_k$ is chosen to be a fast increasing sequence such that
\begin{equation}\label{lim}
\l(\b)=\lim_{k\rightarrow \infty}\frac{t_{m_k}}{m_k},\  \ \  \  \
k+\Gamma_k+1+\sum_{j=1}^{k-1}(m_j+t_{m_j}+1)\ll m_k.
\end{equation}

We can see that $U$ is well defined. This is because, for any $\e=(\e_1,\e_2,...,\e_k)\in \Sigma_{\beta}^k$, the interval $I(\e_1,\e_2,...,\e_k,0^{\G_k+1})$ is full by Theorem \ref{lem3.2}(4) and it follows from Theorem \ref{lem3.2}(1) that $(\e_1,\e_2,...,\e_k,0^{\G_k+1})$ can concatenate any $\b$-admissible word. Analogously, $I(\omega_k)$ is full by  Theorem \ref{lem3.2}(4) and $\w_k$ can concatenate any admissible word by Theorem \ref{lem3.2}(1) for each $k \geq 1$.

Clearly, $U$ is a $G_{\d}$ set since ${\rm int}(I(\e))$ is open. Next we will show that $U$ is a subset of $E$ and is dense in $[0,1]$.
\begin{lemma}\label{le2}
$U$ is dense in $[0,1]$.
\end{lemma}
\begin{proof}
Given $x\in[0,1]$ and $r > 0$, we only need to find $y\in U$ such that $|x-y|\leq r.$ Let the $\b$-expansion of $x$
be $\e(x,\b)=(\e_1(x),\e_2(x),...)$, specially, if $x=0$, let $\e(x,\b)=(0,0,\ldots,0,\ldots)$. Clearly there exists $\ell \in \N$
such that $\b^{-\ell} \leq r$.

Let $\w_k$ be defined as (\ref{wk}). Take\\
$$\begin{array}{rcl}
u_1 & = & (\e_1(x),...,\e_\ell(x),0^{\Gamma_\ell+1},\w_1,...,\w_\ell)\in \Sigma^{\ell_1}_\beta, {\rm where}\ \ell_1= \ell+\Gamma_{\ell}+1+ \sum\limits_{j=1}^{\ell}(m_j+t_{m_j}+1), \\
u_2 & = & (u_1,0^{\Gamma_{\ell_1}+1},\w_1,...,\w_{\ell_1})\in \Sigma^{\ell_2}_\beta, {\rm where}\ \ell_2= {\ell_1}+\Gamma_{{\ell_1}}+1+ \sum\limits_{j=1}^{\ell_1}(m_j+t_{m_j}+1), \\
\vdots\\
u_k & = & (u_{k-1},0^{\Gamma_{\ell_{k-1}}+1},\w_1,...,\w_{\ell_{k-1}})\in \Sigma^{\ell_k}_\beta, {\rm where}\ \ell_k= {\ell_{k-1}}+\Gamma_{{\ell_{k-1}}}+1+ \sum\limits_{j=1}^{\ell_{k-1}}(m_j+t_{m_j}+1),\\
\vdots\\
\end{array}$$ Let $$S := \bigcap_{k=1}^\infty{\rm int}(I(u_k)).$$For each $k\geq 1$, we have$${\rm int}(I(u_{k+1}))\subset \overline{I(u_{k+1})} \subset {\rm int}\left(I(u_k,0^{\Gamma_{\ell_k}+1},\omega_1,\ldots,\omega_{\ell_k-1})\right)\subset  I(u_k)\subset \overline{I(u_k)},$$ where the second inclusion relation follows from Lemma \ref{in}. Thus, it follows that $$\bigcap_{k=1}^\infty{\rm int}(I(u_k))=\bigcap_{k=1}^\infty\overline{I(u_k)}.$$ Note that $\overline{I(u_{k+1})}\subset \overline{I(u_k)}$, it is obvious that  $\bigcap\limits_{k=1}^\infty\overline{I(u_k)}$ is nonempty. The intersection $\bigcap\limits_{k=1}^\infty{\rm int}(I(u_k))$ is therefore nonempty which gives the fact that $S \neq \emptyset$. For every $y \in S$,  we get $y \in U$ by the construction of $U$ and $S$. Moreover, we have
$$|x-y|\leq \b^{-\ell} \leq r$$ since both the $\beta$-expansions of $x$ and $y$ begin with $\e_1(x),...,\e_\ell(x)$. Therefore, $U$ is dense in $[0,1]$.
\end{proof}
\begin{lemma}\label{le3}
$U \subset E$.
\end{lemma}
\begin{proof}
For every $x \in U,$  we only need to prove that there exists a sequence $\{n_k\}_{k\geq 1}$ such that $$\lim_{k\rightarrow\infty} \frac{-\log_\beta |I_{n_k}|}{n_k}=1+\l(\b).$$ Then it follows that $\overline{D}(x)= 1+\l(\b)$ since $\overline{D}(x)\leq 1+\l(\b)$. As a consequence, $U \subset E$.

In fact, for all $x \in U$, by the construction of $U$, there exist infinitely many $k$, such that the $\beta$-expansion of $x$ starts with $\e_1,\ldots,\e_k,0^{\Gamma_k+1},\w_1,\ldots,\w_k$, where $\e_1,\ldots,\e_k\in \Sigma_{\beta}^k$ and $\w_i$ is defined as (\ref{wk}) for all $1\leq i\leq k$. Let $n_k=k+\G_k+1+\sum\limits_{j=1}^{k-1}(m_j+t_{m_j}+1)+m_k$, for convenience, we denote $n_k$ as $n_k=h_k+m_k$ where $h_k=k+\G_k+1+\sum\limits_{j=1}^{k-1}(m_j+t_{m_j}+1) \ll m_k$ by (\ref{lim}).  A simple observation on $I_{n}(x)$: if $n=n_k+t_{m_k}+1$, by the construction of $U$, we have
\begin{equation}\label{ii}
 I_n(x)= I(\e_1,\ldots,\e_{h_k},\e_1^\ast,\ldots,\e_{m_k}^\ast,0^{t_{m_k}+1}),
\end{equation} and $I_n(x)$ is full by Theorem \ref{lem3.2}(4).

Now we estimate the length of $I_n(x)$ when $n=n_k$. By the definition of $t_n$ and the criterion of
admissibility, it follows that
$$I(\e_1^\ast,\ldots,\e_{m_k}^\ast) = I(\e_1^\ast,\ldots,\e_{m_k}^\ast,0^{t_{m_k}}),$$
so
$I_{n_k}(x)=I(\e_1,\ldots,\e_{h_k},\e_1^\ast,\ldots,\e_{m_k}^\ast,0^{t_{m_k}})$ by (\ref{ii}).

Hence, we have $I_{n_k}(x)=I_{n_k+t_{m_k}}(x)$. Since
$I_{n_k+t_{m_k}+1}(x)$ is full, we have
$$\b^{-(n_k+t_{m_k}+1)} \leq |I_{n_k}(x)|=|I_{n_k+t_{m_k}}(x)| \leq
\b^{-(n_k+t_{m_k})}.$$

By (\ref{lim}), it immediately follows that $$\lim_{k\rightarrow\infty} \frac{-\log_\beta |I_{n_k}|}{n_k}=\lim_{k\rightarrow\infty} \frac{n_k+t_{m_k}}{n_k}=\lim_{k\rightarrow\infty} \frac{h_k+m_k+t_{m_k}}{h_k+m_k}=1+\l(\b).$$
\end{proof}

\textbf{Proof of Theorem \ref{r1}}  Since $U$ is dense in $[0,1]$ and it is a $G_\d$ set, we easily get that $U$ is residual in $[0,1]$ by Baire Category Theorem. Moreover, Lemma \ref{le3} ensures that $E$ is residual in $[0,1]$.

$\hfill\Box$

{\bf Acknowledgement}
The authors are grateful to Lingmin Liao for giving useful suggestion. This work was supported by NSFC 11371148 and 11411130372, Guangdong Natural Science Foundation 2014A030313230, and "Fundamental Research Funds for the Central Universities" SCUT 2015ZZ055.

\end{document}